\documentclass[11pt]{amsart}
\usepackage[margin=1in]{geometry}
\usepackage{amsfonts, amstext, amsmath, amssymb}
\usepackage{graphicx}

\newtheorem{lemma}{Lemma}
\newtheorem{theorem}[lemma]{Theorem}
\newtheorem{proposition}[lemma]{Proposition}

\begin{document}

\title{Equivalents of disjunctive Markov's principle}

\author[Hendtlass]{Matthew Hendtlass}
 \address{School of Mathematics and Statistics,
    University of Canterbury,
    Christchurch 8041,
    New Zealand}
\email{matthew.hendtlass@canterbury.ac.nz}

\date{October 31, 2016}

\begin{abstract}
\noindent
The purpose of this short note is to point out a rich source of natural equivalents of the weak semi-intuitionistic principle \textbf{MP}$^\vee$ in reverse constructive mathematics: many simple theorems from Euclidean geometry when read classically (for example with $<$ interpreted as $\leqslant$ and $\neq$) are equivalent to disjunctive Markov's principle \textbf{MP}$^\vee$. We give an example of this phenomenon.
\end{abstract}

\maketitle


\bigskip
\noindent
This paper is a small contribution to constructive reverse mathematics. In constructive reverse mathematics we classify, in particular,\footnote{Any theorem which is independent of the base theory is subject to reverse mathematics; thus constructive reverse mathematics also considers results from Brouwer's intuitionism and the Russian school of recursive mathematics as well as negations of some intuitionistic, recursive, or classical results.} theorems of classical mathematics (\textbf{ZF} with dependent choice, say) by the fragment of the law of excluded middle required to prove them (in addition to constructive techniques). It is similar, and indeed inspired by, Friedman's programme of reverse mathematics \cite{Simpson}; however, since we are interested mostly in logical, rather than set theoretical, principles, we take as our base theory full constructive set theory \textbf{CZF} \cite{AR}\footnote{Since predicativity issues rarely play an integral part in standard mathematics, we may alternatively work with Intuitionistic \textbf{ZF} set theory, which is equiconsistant with \textbf{ZF}.} possibly with some form of choice (typically dependent choice). See \cite{BB, BV} for the development of constructive mathematics (a la Bishop, \textbf{BISH}) and \cite{Ish_CRM} for an overview of results in constructive reverse mathematics together with references. We equate \textbf{BISH} with the mathematics of \textbf{CZF} plus dependent choice, in particular based upon intuitionistic logic.

\bigskip
\noindent
Most weak `semi-intuitionistic principles' were introduced as equivalent, over \textbf{BISH}, to fundamental properties in analysis; for example, Ishihara introduced \textbf{BD}-$\mathbf{N}$ as a logical equivalent of the assertion that every sequentially continuous function is pointwise continuous \cite{Ish92}, and \textbf{WMP} is important for its equivalence with `every mapping from a complete metric space to a metric space is strongly extensional'\footnote{That is, for any function $f$ between metric spaces $X,Y$ and all $x,x^\prime\in X$, if $\rho_Y(f(x),f(x^\prime))>0$, then $\rho(x,x^\prime)>0$.} \cite{Ish92}. In contrast, disjunctive Markov's principle---usually stated in terms of binary sequences as
 \begin{quote}
  \textbf{MP}$^\vee$: if $\alpha$ is a binary sequence with at most one nonzero term and such that it is not the case that all terms are $0$, then either all the even terms are zero or all the odd terms are zero
 \end{quote}
\noindent
---is introduced as a common weakening of Markov's principle--- 
 \begin{quote}
  \textbf{MP}: If $\alpha$ is a binary sequence such that it is impossible for every term to be $0$, then there exists $n$ such that $\alpha(n)=1$
 \end{quote}
 \noindent
---and the lesser limited principle of omniscience, and which together with weak Markov's principle is equivalent to the full form of Markov's principle.\footnote{Both \textbf{WMP} and \textbf{MP}$^\vee$ are independent of \textbf{IZF}; see for example \cite{HL}.}

\bigskip
\noindent
There are few equivalents of disjunctive Markov's principal in the literature, probably the most natural of which is the following result due to Mandelkern \cite{Mandelkern}. 
 \begin{quote}
  \textbf{MP}$^\vee$ is equivalent to the statement `if $x,y\in\mathbf{R}$ are such that $\neg\neg(x<y)$, then $\{x,y\}$ is closed.'
 \end{quote}
 \noindent
With a weak form of countable choice (see \cite{BRS}), \textbf{MP}$^\vee$ is equivalent to 
 \begin{quote}
  \textbf{MP}$^\vee_\mathbf{R}$: if $x$ is a real number such that $\neg(x=0)$, then either $x\leqslant0$ or $x\geqslant 0$.
 \end{quote}
\noindent
It is this form of \textbf{MP}$^\vee$ we shall use; if we replace \textbf{MP}$^\vee$ by \textbf{MP}$^\vee_\mathbf{R}$ in what follows (and \textbf{MP} by `if $a$ is a real number such that $\neg(a=0)$, then $|a|>0$), then our results do not require any form of the axiom of choice (they are valid in \textbf{CZF}). \textbf{MP}$^\vee_\mathbf{R}$ fails in the sheaf model of continuous functions over the reals (see \cite{Grayson}).

\bigskip
\noindent
We argue that \textbf{MP}$^\vee$ is precisely what is required to extract algorithms from a particular class of classical results in elementary Euclidean geometry. 
In order to introduce our example we need a few definitions. A \emph{polygon} is given by a finite sequence $x_0,\ldots,x_k$ of vertices in $\mathbf{R}^2$ such that $x_0=x_k$ and no two \emph{edges}---line segments (without endpoints) joining two consecutive elements of the sequence---intersect. A \emph{strictly convex polygon} is a polygon such that for any three consecutive vertices $x_i,x_{i+1},x_{i+2}$, the internal angle $\widehat{x_ix_{i+1}x_{i+2}}$ is less than $\pi$. We associate the polygon given by $x_0,\ldots,x_k$ with the closure of the interior of some Jordan curve which traces the edges (this is constructively well defined, in particular see \cite{JCT} for a constructive treatment of the Jordan curve theorem). 

\begin{theorem}
\label{1}
For any finite collection $S$ of more than three points in $\mathbf{R}^2$, if no three points are collinear, then there exists a strictly convex polygon with vertices from $S$ and which contains $S$.
\end{theorem}

\noindent
The constructive status of Theorem \ref{1} depends on the interpretations of `not collinear' and $<$ (the latter via the definition of strictly convex polygon). We may define `not collinear' in either a positive way
 \begin{quote}
  three points are \emph{non-collinear} if there exists $\varepsilon>0$ such that each point is bound away from the line through the other two points by at least epsilon,
 \end{quote}
\noindent
or the negative way
 \begin{quote}
  three points are not collinear if it is not the case that they are collinear.
 \end{quote}
\noindent
Similarly, we can give $<$ a strong, positive definition---$x<y$ if there exists a positive rational $r$ such that $|x-y|>r$, which is the definition we adopt---or the classically equivalent weak definition $x<_wy$ if $\neq(x\geqslant y)$.

\bigskip
\noindent
We isolate three versions of this theorem:
 \begin{itemize}
  \item[1.] the constructive version: we use both positive definitions;
  \item[2.] the classical version: we take both negative definitions;
  \item[3.] an incongruous version: the have the negative definition in the antecedent (that is, of not collinear) and the positive definition, of $<$, in the consequent.
 \end{itemize}
\noindent
It turns out the classical version is equivalent to \textbf{MP}$^\vee$; whence disjunctive Markov's principle is precisely what is required to extract a valid algorithm from an appropriate classical proof of this theorem. The constructive version is, as the names suggests, fully constructive, and the incongruous version is equivalent to the full form of Markov's principle.

\bigskip
\noindent
We need some basic definitions and notation. For convenience we adopt the classical definition of distinct: elements $x,y$ of some set with equality are \emph{distinct} if $\neg(x=y)$; in a metric space this is weaker than inequality, $x\neq y$ if $\rho(x,y)>0$. For two distinct points $a,b\in\mathbf{R}^2$, we denote by $L_{a,b}$ the \emph{line passing through} both $a,b$, and we denote by $[a,b]$ the \emph{interval} $\{ta+(1-t)b:t\in[0,1]\}$ from $a$ to $b$. If a line  $L$ is bounded away from the origin, then we write $L^+$ for the open half space defined by $L$ which contains the origin and $L^-$ for the other open half space. If $x_0,\ldots,x_k$ describes a convex polygon $P$ which contains the origin, then 
 $$
  P=\overline{\bigcap_{i=0}^{k-1}L_{x_i,x_{i+1}}^+}.
 $$

\bigskip
\noindent
A subset $S$ of a metric space is \emph{located} if for all $x\in X$ the \emph{distance}
 $$
  \rho(x,S)=\inf\{\rho(x,s):s\in S\}
 $$
\noindent
exists; the line $L_{a,b}$ is located for any $a,b\in\mathbf{R}^2$. The \emph{metric complement} $-S$ of $S$ is the set $\{x\in X:\rho(x,S)>0\}$ of elements of $X$ which are bounded away from $S$. We denote by $\overline{S}$ the closure of $S$ and by $B(x,\varepsilon)$ the open ball centred on $x$ with radius $\varepsilon$. With these definitions, $x,y,z$ are non-collinear if $x\in-L_{y,z}$. In the constructive formulation of the above theorem (Proposition \ref{1con}) the non-collinear condition guarantees that distinct points of $S$ are indeed bounded apart: for all $x,y\in S$, if $\neg(x=y)$, then $x\neq y$.

\begin{lemma}
\label{lemma}
Let $a,b\in\mathbf{R}^2$ be such that $0\in -L_{a,b}$. Then $-L_{a,b}=L_{a,b}^+\cup L_{a,b}^-$, and if \textbf{MP}$^\vee$ holds, then $R^2=\overline{L_{a,b}^+}\cup \overline{L_{a,b}^-}$.
\end{lemma}

\begin{proof}
The first statement is immediate; the second follows since, for an appropriate oriented vector $u$ orthogonal to $L_{a,b}$, $x\in \overline{L_{a,b}^+}$ if and only if $x=y+ru$ for some $y\in L_{a,b}$ and some $r\geqslant0$ and $x\in \overline{L_{a,b}^-}$ if and only if $x=y+ru$ for some $y\in L_{a,b}$ and some $r\leqslant0$.
\end{proof}

\bigskip
\noindent
We begin with the constructively valid version of Theorem \ref{1}. 

\begin{proposition}
\label{1con}
For any finite collection $S$ of at least three points in $\mathbf{R}^2$, if $x\in-L_{y,z}$ for any distinct points $x,y,z\in S$, then there exists a strictly convex polygon with vertices from $S$ and which contains $S$.
\end{proposition}

\begin{proof}
We may assume that any three points of $S\cup\{(0,0)\}$ are non-collinear. Note that, since any three distinct points in $S$ are non-collinear, equality is decidable on $S$. Let $N>0$ be such that $S\subset B(0,N)$ and let $\varepsilon>0$ be such that $\|x-y\|>\varepsilon$ for all distinct $x,y\in S$ (such an $\varepsilon$ exists by our positive non-collinear condition). Set $\delta=\sqrt{N^2-\varepsilon^2}-N$ and let $x_0\in S$ be such that
 \[
 \max\{|s|:s\in S\}-|x_0|<\delta;
 \]
\noindent
denote by $L$ the line perpendicular to $x_0$ which passes through $x_0$ with some fixed orientation. Using our positive formulation of not collinear, we can find $x_1\in S$ such that the angle $\theta$ between $L$ and $[x_0,x_1]$ is minimal. We show that $S\setminus\{x_0,x_1\}$ is contained in $L_{x_0,x_1}^+$. It follows from our choice of $\delta$ that $S$ is contained in $\overline{L^+}$, so any point of $S$ contained in $L_{x_0,x_1}^-$ must be in the wedge between $L$ and the ray from $x_0$ through $x_1$. This contradicts the construction of $x_1$ to minimise $\theta$; thus, by Lemma \ref{lemma}, $S\setminus\{x_0,x_1\}$ is contained in $L_{x_1,x_1}^+$. We can now repeat the construction of $x_1$ with $L$ replaced by $L_{x_0,x_1}$, oriented from $x_0$ to $x_1$, to construct a $x_2$ such that $S\setminus\{x_0,x_1,x_2\}$ is contained within $\widehat{x_0x_1x_2}$. The angle $\widehat{x_0x_1x_2}$ is less than $\pi$ since $c\in L^+_{a,b}$. Since this contradicts the construction of $x_1$ to minimise $\theta$, by Lemma \ref{lemma}, $S\setminus\{x_0,x_1\}$ is contained in $L_{x_1,x_1}^+$. We can now repeat the construction of $x_1$ with $L$ replaced by $L_{x_0,x_1}$, oriented from $x_0$ to $x_1$, to construct $x_2$ such that $S\setminus\{x_0,x_1,x_2\}$ is contained within $\widehat{x_0x_1x_2}$. The angle $\widehat{x_0x_1x_2}$ is less than $\pi$ since $c\in L^+_{a,b}$. 

\bigskip
\noindent
Continuing in this manner we construct a finite sequence $x_0,\ldots, x_{|S|}$ such that $\widehat{x_ix_{i+1}x_{i+2}}<\pi$ for each $0\leqslant i\leqslant |S|-2$. We claim that $x_0=x_k$ for some $k<|S|$. For otherwise, by the pigeon-hole principle, some other element $s$ occurs twice. Let $x_i,x_j$ ($i<j$) be two consecutive occurrences of $s$; then $x_i,\ldots, x_j$ is, by construction, a strictly convex polygon which does not contain $x_0$---this is a contradiction since each internal angle of $x_i,\ldots, x_j$ is maximal by construction. Let $k>0$ be minimal such that $x_0=x_k$. Without loss of generality, the convex hull of $\{x_0,\ldots,x_{k-1}\}$ contains the origin. Then $x_0,\ldots,x_i$ is the desired polygon; the construction guarantees that this polygon contains $S$ since it is equal to the closure of
 $$
  \bigcap_{i=0}^{k-1}L_{x_i,x_{i+1}}^+.
 $$
\end{proof}

\bigskip
\noindent
A polygon $x_0,\ldots,x_k$ is \emph{almost strictly convex} if for any three consecutive vertices $x_i,x_{i+1},x_{i+2}$, $\neg(\widehat{x_ix_{i+1}x_{i+2}}\geqslant\pi)$. The proof of the classical version from \textbf{MP}$^\vee$ requires no new ideas.

\begin{proposition}
\label{1class}
The following are equivalent.
\begin{itemize}
\item[(i)]\textbf{MP}$^\vee$.
\item[(ii)] For any finite collection of more than three points in $\mathbf{R}^2$, if no three points are collinear, then there exists an almost strictly convex polygon  with vertices from $S$ and which contains $S$.
\end{itemize}
\end{proposition}

\begin{proof}
Suppose that \textbf{MP}$^\vee$ holds. The proof of Proposition \ref{1con}, together with the second part of Lemma \ref{lemma}, shows that such a polygon cannot fail to exist; Lemma \ref{lemma} together with \textbf{MP}$^\vee$ allows us to find this polygon using an exhaustive search.\footnote{Alternatively, a more direct algorithm can be given along the lines of that in Proposition \ref{1con}.}

\bigskip
\noindent
For the converse, let $a\in\mathbf{R}$ be such that $\neg(a=0)$ and consider the subset 
 $$
  S=\{(-1,-1),(-1,1),(1,1),(1,-1),(1+a,0)\}.
 $$ 
If $a>0$, then the almost strictly convex polygon must have each element of $S$ as a vertex, and if $a<0$ the set of vertices must be $S\setminus\{(1+a,0)\}$. Hence by counting the vertex set of the polygon from the conclusion of (ii) we can conclude either that $a\geqslant 0$, if there are five vertices, or $a\leqslant0$ if there are four vertices.
\end{proof}

\bigskip
\noindent
The source of disjunctive Markov's principle in the above argument is simple. \textbf{MP}$^\vee$ is equivalent to `for all $x,y,z$ in $\mathbf{R}^2$ which are not collinear and any normal $\mathbf{n}$ to the line $L$ through $y$ and $z$, either $r\leqslant0$ or $r\geqslant0$ for the unique decomposition of $x$ into a point on $L$ plus $r\mathbf{n}$.' Any argument which contains this result will require \textbf{MP}$^\vee$. In particular, the classical formulations of many simple geometric results (like the above) that one might expect to be constructive. This also applies to higher dimensions 
and separations using hyperplanes.

\bigskip
\noindent
The characterisation of the final version is again similar; Markov's principle is really only required to prove the angle is indeed less than $\pi$.

\begin{proposition}
\label{1inc}
The following are equivalent.
\begin{itemize}
\item[(i)] \textbf{MP}.
\item[(ii)] For any finite collection of more than three points in $\mathbf{R}^2$, if no three points are collinear, then there exists a strictly convex polygon  with vertices from $S$ and which contains $S$.
\end{itemize}
\end{proposition}

\begin{proof}
The direction from (i) to (ii) follows from Proposition \ref{1con}. The converse is similar to the previous theorem except we now consider the set $S=\{(-1,-1),(-1,1),(1,1),(1,-1),(1+a,0),(-1+a,0)\}$. The strictly convex polynomial must contain either $(1+a,0)$ or $(-1+a,0)$, and since the internal angle at this node is less than $\pi$, we can bound $a$ away from $0$.
\end{proof}

\end{document}